\documentclass[10pt]{amsart}
\usepackage{amssymb}
\usepackage{bm}
\usepackage{graphicx}
\usepackage[centertags]{amsmath}
\usepackage{amsfonts}
\usepackage{amsthm}
\usepackage{amsbsy}
\usepackage{mathtools}
\usepackage{mathrsfs}
\usepackage[all]{xy}
\newtheorem{thm}{Theorem}
\newtheorem{cor}[thm]{Corollary}
\newtheorem{lem}[thm]{Lemma}

\newtheorem{prop}[thm]{Proposition}
\newtheorem{claim}[thm]{Claim}

\newtheorem*{Th1'}{Theorem 1$'$}

\theoremstyle{definition}

\newcommand{\rr}{\mathbb{R}}

\newcommand{\ee}{\varepsilon}

\newcommand{\meg}{\geqslant}
\newcommand{\mik}{\leqslant}
\newcommand{\ave}{\mathbb{E}}

\newcommand{\bbo}{\boldsymbol{\Omega}}
\newcommand{\cala}{\mathcal{A}}
\newcommand{\calf}{\mathcal{F}}
\newcommand{\cals}{\mathcal{S}}
\newcommand{\bcalf}{\bm{\mathcal{F}}}

\newcommand{\bbp}{\mathbf{P}}
\newcommand{\bx}{\mathbf{x}}
\newcommand{\by}{\mathbf{y}}
\newcommand{\bz}{\mathbf{z}}

\begin{document}

\title{A concentration inequality for product spaces}

\author{Pandelis Dodos, Vassilis Kanellopoulos and Konstantinos Tyros}

\address{Department of Mathematics, University of Athens, Panepistimiopolis 157 84, Athens, Greece}
\email{pdodos@math.uoa.gr}

\address{National Technical University of Athens, Faculty of Applied Sciences,
Department of Mathematics, Zografou Campus, 157 80, Athens, Greece}
\email{bkanel@math.ntua.gr}

\address{Mathematics Institute, University of Warwick, Coventry, CV4 7AL, UK}
\email{k.tyros@warwick.ac.uk}

\thanks{The third author was supported by ERC grant 306493.}

\maketitle


\begin{abstract}
We prove a concentration inequality which asserts that, under some mild regularity conditions, every random
variable defined on the product of sufficiently many probability spaces exhibits pseudorandom behavior.
\end{abstract}


\section{Introduction}

\numberwithin{equation}{section}

Our goal in this paper is to prove a concentration inequality for product spaces which is somewhat different in spirit when compared
with the well-known concentration inequalities discovered by Talagrand \cite{Tal1,Tal2}. Roughly speaking, it asserts that under
some mild  regularity conditions, every random variable defined on the product of sufficiently many probability spaces exhibits
pseudorandom behavior.

To state this inequality we need to introduce some pieces of notation. Let $n$ be a positive integer and let
$(\Omega_1,\calf_1,\mathbb{P}_1),\dots,(\Omega_n,\calf_n,\mathbb{P}_n)$ be a finite sequence of probability spaces.
By $(\bbo,\bcalf,\bbp)$ we denote their product. More generally, for every nonempty subset $I$ of $\{1,\dots,n\}$ by
$(\bbo_I,\bcalf_{\!I},\mathbf{P}_{\!I})$ we denote the product of the spaces $\langle(\Omega_i,\calf_i,\mathbb{P}_i):i\in I\rangle$.
In particular, we have
\begin{equation} \label{e1.1}
\boldsymbol{\Omega}=\prod_{i=1}^n\Omega_i \ \text{ and } \ \boldsymbol{\Omega}_I=\prod_{i\in I} \Omega_i.
\end{equation}
(By convention, $\boldsymbol{\Omega}_{\emptyset}$ stands for the empty set.)

Now let $f\colon\boldsymbol{\Omega}\to\rr$ be an integrable random variable and $I\subseteq \{1,\dots,n\}$ such that $I$ and
$I^{\mathsf{c}}\coloneqq \{1,\dots,n\}\setminus I$ are nonempty. For every $\mathbf{x}\in\boldsymbol{\Omega}_I$ let
$f_{\mathbf{x}}\colon\boldsymbol{\Omega}_{I^{\mathsf{c}}}\to\rr$ be the section of $f$ at $\bx$, that is,
$f_{\mathbf{x}}(\mathbf{y})=f\big( (\mathbf{x},\mathbf{y})\big)$ for every $\by\in\bbo_{I^{\mathsf{c}}}$.
Fubini's theorem asserts that the random variable $\mathbf{x}\mapsto \ave(f_{\mathbf{x}})$ is integrable and satisfies
\begin{equation} \label{e1.2}
\int \ave(f_{\mathbf{x}}) \, d\mathbf{P}_{\!I} = \ave(f).
\end{equation}
Beyond this basic information, not much can be said at this level of generality. This random variable is rather amorphous.

However, our main result shows that if $f\in L_p(\bbo,\bcalf,\bbp)$ for some $p>1$ and $n$ is sufficiently large,
then one can find a set $I$ of coordinates of cardinality proportional to $n$, such that the random variable
$\bbo_I\ni \mathbf{x}\mapsto \ave(f_{\mathbf{x}})$ is highly concentrated around its mean. Specifically, we have the following theorem.
\begin{thm} \label{t1}
Let $0<\ee\mik 1$ and $1<p\mik 2$, and set
\begin{equation} \label{e1.3}
c(\ee,p)= \frac{1}{4}\,\ee^{\frac{2(p+1)}{p}} (p-1).
\end{equation}
Also let $n$ be a positive integer with $n\meg 2/c(\ee,p)$ and let $(\bbo,\bcalf,\bbp)$ be the product of a finite
sequence $(\Omega_1,\calf_1,\mathbb{P}_1),\dots,(\Omega_n,\calf_n,\mathbb{P}_n)$ of probability spaces. Then for every
$f\in L_p(\bbo,\bcalf,\bbp)$ with $\|f\|_{L_p}\mik 1$ there exists an interval $J$ of\, $\{1,\dots,n\}$
with $J^{\mathsf{c}}\neq\emptyset$ and
\begin{equation} \label{e1.4}
|J|\meg c(\ee,p) \, n
\end{equation}
such that for every nonempty $I\subseteq J$ we have
\begin{equation} \label{e1.5}
\mathbf{P}_{\!I} \big( \{ \mathbf{x}\in\boldsymbol{\Omega}_I: |\ave(f_{\bx}) - \ave(f)|\mik \ee \}\big) \meg 1-\ee.
\end{equation}
\end{thm}
Of course, the case of random variables in $L_p(\bbo,\bcalf,\bbp)$ for $p>2$ is reduced to the case $p=2$. In other words, Theorem \ref{t1}
is valid for any $p>1$. Also notice that Theorem 1 can be reformulated as follows.
\begin{Th1'}
Let $\ee, p, n$ be as in Theorem \emph{\ref{t1}} and let $X_1,\dots,X_n$ be a finite sequence of independent random variables defined
on a probability space $(\Omega,\mathcal{F},\mathbb{P})$. Let $Y$ be another random variable which can be expressed as
$Y=F(X_1,\dots,X_n)$ for some measurable function $F$, and assume that $\ave(|Y|^p)\mik 1$. Then there exists an interval $J$ of\,
$\{1,\dots,n\}$ with $J^{\mathsf{c}}\neq\emptyset$ and satisfying \eqref{e1.4}, such that for every nonempty $I\subseteq J$ we have
\begin{equation} \label{e1.6}
\mathbb{P}\big( |\ave(Y\, |\, \mathcal{F}_I)-\ave(Y)|\mik \ee \big)\meg 1-\ee
\end{equation}
where\, $\ave(Y\, |\, \mathcal{F}_I)$ stands for the conditional expectation of\, $Y$ with respect to the
$\sigma\text{-algebra}$ $\mathcal{F}_I\coloneqq \sigma\big(\{X_i: i\in I\}\big)$.
\end{Th1'}
We proceed to discuss another consequence of Theorem \ref{t1} which is of ``geometric" nature. Let $\bbo$ be as in Theorem \ref{t1}
and let $A$ be a measurable event of $\bbo$. Also let $I\subseteq\{1,\dots,n\}$ such that $I$ and $I^{\mathsf{c}}$ are nonempty, and
observe that if $f$ is the indicator function of $A$, then for every $\bx\in\bbo_I$ the quantity $\ave(f_{\bx})$ is the probability of the section
$A_{\bx}=\{\by\in\bbo_{I^{\mathsf{c}}}: (\bx,\by)\in A\}$ of $A$ at $\bx$. Taking into account this remark, we obtain the following corollary.
\begin{cor} \label{c2}
Let $\ee, p, n$ and $(\bbo,\bcalf,\bbp)$ be as in Theorem \emph{\ref{t1}}. Then for every $A\in\bcalf$ there exists an interval $J$
of\, $\{1,\dots,n\}$ with $J^{\mathsf{c}}\neq\emptyset$ and satisfying \eqref{e1.4}, such that for every nonempty $I\subseteq J$ we have
\begin{equation} \label{e1.7}
\mathbf{P}_{\!I} \Big( \big\{ \mathbf{x}\in\boldsymbol{\Omega}_I: |\bbp_{\!I^{\mathsf{c}}}(A_{\bx}) - \bbp(A)|\mik \ee \bbp(A)^{1/p}
\big\}\Big) \meg 1-\ee.
\end{equation}
\end{cor}
Versions of Corollary \ref{c2} for subsets of the product of certain finite probability spaces were proved in \cite{DKT1,DKT2}
and were applied to combinatorial problems (we will briefly comment on these applications in Subsection 4.1, and for a more complete
exposition we refer the reader to \cite{DK}). Theorem \ref{t1} was motivated by these results and was found in an effort to abstract
their probabilistic features. We expect that Theorem \ref{t1} will in turn facilitate further applications, possibly even beyond the
combinatorial context of \cite{DKT1,DKT2}.

We also note that Corollary \ref{c2} does not hold true for $p=1$ (thus, the range of $p$ in Theorem \ref{t1} is optimal).
To see this, let $n$ be an arbitrary positive integer and for every $i\in \{1,\dots,n\}$ let $(\Omega_i,\calf_i,\mathbb{P}_i)$
be a probability space with the property that there exists a measurable event $A_i$ of $\Omega_i$ with $\mathbb{P}_i(A_i)=1/2$.
As above, we denote by $(\bbo,\bcalf,\bbp)$ the product of the spaces
$(\Omega_1,\calf_1,\mathbb{P}_1),\dots,(\Omega_n,\calf_n,\mathbb{P}_n)$ and we set $A=A_1\times\cdots\times A_n\in \bcalf$.
Notice that if $I$ is a subset of $\{1,\dots,n\}$ such that $I$ and $I^{\mathsf{c}}$ are nonempty, then for every $\bx\in\bbo_I$
we have $\bbp_{\!I^{\mathsf{c}}}(A_{\bx})=0$ if $\bx\notin\prod_{i\in I} A_i$ while $\bbp_{\!I^{\mathsf{c}}}(A_{\bx})=2^{-n+|I|}$
if $\bx\in\prod_{i\in I} A_i$. It follows, in particular, that for every $\bx\in\bbo_I$ we have
$|\bbp_{\!I^{\mathsf{c}}}(A_{\bx})-\bbp(A)|\meg \bbp(A)$ and so if $p=1$, then the probability of no section of $A$ can approximate
the probability of $A$ with the desired accuracy.

Some final remarks on the proof of Theorem \ref{t1} which is based on a certain estimate for martingale difference sequences.
Martingales are, of course, very useful tools for obtaining concentration inequalities (see, e.g., \cite{L,MS} and the references
therein). However, the most interesting part of the argument is how one locates the desired interval $J$. This is achieved
with a variant of Szemer\'{e}di's regularity lemma~\cite{Sz}, especially as described by Tao in \cite{Tao}.

\subsubsection*{Acknowledgments}
We would like to thank the anonymous referees for their comments and remarks, and for suggesting Theorem 1$'$.


\section{An estimate for martingale difference sequences}

Recall that a finite sequence $(d_i)_{i=1}^n$ of random variables is said to be a \emph{martingale difference sequence} if it is of the form
\begin{equation} \label{e2.1}
d_i=f_i-f_{i-1}
\end{equation}
where $(f_i)_{i=1}^n$ is a martingale and $f_0=0$. Clearly, for any $p\meg 1$, every martingale difference sequence in $L_p$ is a monotone basic
sequence. Also notice that martingale difference sequences are orthogonal in $L_2$. Hence, for every martingale difference sequence
$(d_i)_{i=1}^n$ in $L_2$ we have
\begin{equation} \label{e2.2}
\Big( \sum_{i=1}^n \|d_i\|^2_{L_2} \Big)^{1/2} = \big\| \sum_{i=1}^n d_i\big\|_{L_2}.
\end{equation}
We will need the following extension of this basic fact.
\begin{prop} \label{p3}
Let $(\Omega,\calf,\mathbb{P})$ be a probability space and $1<p \mik 2$. Then for every martingale difference
sequence $(d_i)_{i=1}^n$ in $L_p(\Omega,\calf,\mathbb{P})$ we have
\begin{equation} \label{e2.3}
\Big( \sum_{i=1}^n \|d_i\|^2_{L_p} \Big)^{1/2} \mik \Big(\frac{1}{p-1}\Big)^{1/2} \cdot \big\| \sum_{i=1}^n d_i\big\|_{L_p}.
\end{equation}
\end{prop}
The estimate in \eqref{e2.3} is optimal, and was recently proved by Ricard and Xu \cite{RX} who deduced it (via an elegant pseudo-differentiation argument)
from the following sharp uniform convexity inequality for $L_p$ spaces.
\begin{prop}[\cite{BCL}] \label{p4}
Let $(\Omega,\Sigma,\mu)$ be an arbitrary measure space and $1<p \mik 2$. Then for every $x,y\in L_p(\Omega,\Sigma,\mu)$ we have
\begin{equation} \label{e2.4}
\|x\|_{L_p}^2 + (p-1) \|y\|_{L_p}^2 \mik \frac{\|x+y\|_{L_p}^2+\|x-y\|_{L_p}^2}{2}.
\end{equation}
\end{prop}
For details, as well as noncommutative extensions, we refer to \cite{RX}.


\section{Proof of Theorem \ref{t1}}

Let $n$ be a positive integer and let $(\Omega_1,\calf_1,\mathbb{P}_1),\dots,(\Omega_n,\calf_n,\mathbb{P}_n)$ be a finite
sequence of probability spaces. Recall that by $(\bbo,\bcalf,\bbp)$ we denote their product.
For notational simplicity, for every $m\in \{1,\dots,n\}$ we shall denote by $(\bbo_m,\bcalf_m,\bbp_m)$ the
product of the spaces $(\Omega_1,\calf_1,\mathbb{P}_1),\dots,(\Omega_m,\calf_m,\mathbb{P}_m)$.
Notice that the $\sigma$-algebra $\bcalf_m$ is not comparable to $\bcalf$, but it may be ``lifted" to the full
product $\bbo$ using the natural projection $\pi_m\colon \bbo\to \bbo_m$. Specifically, for every $m\in \{1,\dots,n\}$ we set
\begin{equation} \label{e3.1}
\cals_m=\{ \pi_m^{-1}(A): A\in\bcalf_m\}.
\end{equation}
Observe that $\cals_m=\{ A\times \Omega_{m+1}\times\cdots\times \Omega_n: A\in\bcalf_m\}$ if $m<n$ while $\cals_n=\bcalf$.
It follows, in particular, that $(\cals_m)_{m=1}^n$ is an increasing sequence of sub-$\sigma$-algebras of $\bcalf$, and so
for every $1<p\mik 2$ and every $f\in L_p(\bbo,\bcalf,\bbp)$ with $\|f\|_{L_p}\mik 1$ the sequence
$\ave(f \, | \, \cals_1),\dots, \ave(f \, | \, \cals_n)$ is a finite martingale which is contained in the unit ball of
$L_p(\bbo,\bcalf,\bbp)$. We have the following property which is satisfied by all finite martingales of this form.
\begin{lem} \label{l5}
Let $0<\theta\mik 1$ and $1<p\mik 2$, and let $n$ be a positive integer with
\begin{equation} \label{e3.2}
n\meg 8\,\theta^{-2}(p-1)^{-1}.
\end{equation}
Also let $(\Omega, \calf,\mathbb{P})$ be a probability space and $(\cala_m)_{m=1}^n$ an increasing finite sequence
of $\text{sub-}\sigma\text{-algebras}$ of $\calf$. Finally, let $g\in L_p(\Omega,\calf,\mathbb{P})$ with $\|g\|_{L_p}\mik 1$.
Then there exist $i,j\in\{1,\dots,n-1\}$ with
\begin{equation} \label{e3.3}
j-i\meg \big(4^{-1}\theta^2 (p-1)\big)\, n
\end{equation}
such that
\begin{equation} \label{e3.4}
\|\ave(g \, | \, \cala_j)-\ave(g \, | \, \cala_i)\|_{L_p}\mik \theta.
\end{equation}
In particular, for every $m,l\in \{i,\dots,j\}$ we have $\|\ave(g \, | \, \cala_m)-\ave(g \, | \, \cala_l)\|_{L_p}\mik 2\theta$.
\end{lem}
The argument in the proof of Lemma \ref{l5} is, essentially, the $L_p$-version of the ``energy increment strategy" devised
in the proof of Theorem 2.11 in \cite{Tao}. Further applications of this $L_p$-method are given in \cite{DKK}.
\begin{proof}[Proof of Lemma \emph{\ref{l5}}]
We argue by contradiction. So, assume that for every pair $i,j\in \{1,\dots,n-1\}$ satisfying \eqref{e3.3} we have that
$\|\ave(g \, | \, \cala_j)-\ave(g \, | \, \cala_i)\|_{L_p} >\theta$. Set $\ell=\lfloor \theta^{-2}(p-1)^{-1}\rfloor+1$ and notice
that $\lfloor (n-2)/\ell\rfloor\meg 1$. Moreover, for every $k\in \{1,\dots,\ell+1\}$ let $i_k=(k-1)\lfloor (n-2)/\ell\rfloor+1$.
With these choices, for every $k\in\{1,\dots,\ell\}$ we have $1\mik i_k < i_{k+1}\mik n-1$ and
\begin{equation} \label{e3.5}
i_{k+1}-i_k = \Big\lfloor \frac{n-2}{\ell}\Big\rfloor \meg \frac{n}{2\ell} \meg \Big( \frac{\theta^2 (p-1)^2}{4} \Big) n
\end{equation}
which implies, by our assumption that the lemma is false, that
\begin{equation} \label{e3.6}
\|\ave(g \, | \, \cala_{i_{k+1}})-\ave(g \, | \, \cala_{i_k})\|_{L_p} >\theta.
\end{equation}
We set $d_1=\ave(g \, | \, \cala_{i_1})$ and $d_{k+1}=\ave(g \, | \, \cala_{i_{k+1}})-\ave(g \, | \, \cala_{i_k})$
for every $k\in\{1,\dots,\ell\}$, and we observe that the sequence $(d_k)_{k=1}^{\ell+1}$ is a martingale difference sequence
in $L_p(\Omega,\calf,\mathbb{P})$. Therefore, by Proposition \ref{p3}, we obtain that
\begin{eqnarray} \label{e3.7}
\ \ \ \ \ 1 \!\!\!\! & < & \sqrt{p-1}\theta\sqrt{\ell} \ \stackrel{\eqref{e3.6}}{<} \sqrt{p-1} \Big(\sum_{k=1}^{\ell}
\| \ave(g \, | \, \cala_{i_{k+1}})-\ave(g \, | \, \cala_{i_k})\|_{L_p}^2\Big)^{1/2} \\
& \mik & \sqrt{p-1} \Big(\sum_{k=1}^{\ell+1} \|d_k\|_{L_p}^2\Big)^{1/2} \stackrel{\eqref{e2.3}}{\mik}
\big\| \sum_{k=1}^{\ell+1} d_k \big\|_{L_p} = \|\ave(g \, | \, \cala_{i_{\ell+1}})\|_{L_p}  \mik \|g\|_{L_p} \nonumber
\end{eqnarray}
which contradicts, of course, our hypothesis that $\|g\|_{L_p}\mik 1$.

Finally, let $1\mik i < j\mik n$ and notice that for every $i\mik l\mik m \mik j$ we have
\[ \ave(g \, | \, \cala_m)-\ave(g \, | \, \cala_l)\! =\! \ave(\ave(g \, | \, \cala_j)-\ave(g \, | \, \cala_i)\, | \, \cala_m) -
\ave(\ave(g \, | \, \cala_j)-\ave(g \, | \, \cala_i)\, | \, \cala_l) \]
which yields that $\|\ave(g \, | \, \cala_m)-\ave(g \, | \, \cala_l)\|_{L_p}\mik 2 \|\ave(g \, | \, \cala_j)-\ave(g \, | \, \cala_i)\|_{L_p}$.
The proof of Lemma \ref{l5} is completed.
\end{proof}
We will also need the following lemma. In its proof, and in the rest of this paper, we will follow the common practice
when proving inequalities and we will ignore measurability issues since they can be resolved with standard arguments.
\begin{lem} \label{l6}
Let $n$ be a positive integer and $(\Omega_1,\calf_1,\mathbb{P}_1),\dots,(\Omega_n,\calf_n,\mathbb{P}_n)$ a finite sequence
of probability spaces, and denote by $(\bbo,\bcalf,\bbp)$ their product. Also let $I\subseteq \{1,\dots,n\}$ and assume that
$I$ and $I^{\mathsf{c}}$ are nonempty. Then for every $p\meg 1$ and every $g,h\in L_p(\bbo,\bcalf,\bbp)$ we have
\begin{equation} \label{e3.8}
\int \|g_{\bx}-h_{\bx}\|_{L_1}^p \, d\bbp_{\!I} \mik \|g-h\|_{L_p}^p.
\end{equation}
\end{lem}
\begin{proof}
Notice first that, by Fubini's theorem,
\begin{equation} \label{e3.9}
\|g-h\|_{L_p}^p = \int \Big( \int |g_{\bx}-h_{\bx}|^p \, d\bbp_{\!I^{\mathsf{c}}} \Big) \, d\bbp_{\!I}.
\end{equation}
On the other hand, by Jensen's inequality, for every $\bx\in\bbo_I$ we have
\begin{equation} \label{e3.10}
\|g_{\bx}-h_{\bx}\|_{L_1}^p = \Big( \int |g_{\bx}-h_{\bx}| \, d\bbp_{\! I^{\mathsf{c}}} \Big)^p \mik
\int |g_{\bx}-h_{\bx}|^p \, d\bbp_{\! I^{\mathsf{c}}}
\end{equation}
and so, taking the average over all $\bx\in \bbo_I$ and using \eqref{e3.9}, we obtain the desired estimate.
\end{proof}
We are ready to complete the proof of Theorem \ref{t1}.
\begin{proof}[Proof of Theorem \emph{\ref{t1}}]
We fix $f\in L_p(\bbo,\bcalf,\bbp)$ with $\|f\|_{L_p}\mik 1$ and we set
\begin{equation} \label{e3.11}
\theta=\ee^{\frac{p+1}{p}}.
\end{equation}
Since $n\meg 2/c(\ee,p)$, by \eqref{e1.3} and \eqref{e3.11}, we see that $n\meg 8\, \theta^{-2}(p-1)^{-1}$.
Hence, by Lemma \ref{l5} applied to the random variable $f$ and the filtration $(\cals_m)_{m=1}^n$,
there exist $i,j\in\{1,\dots,n-1\}$ satisfying \eqref{e3.3} and such that
\begin{equation} \label{e3.12}
\|\ave(f \, | \, \cals_j)-\ave(f \, | \, \cals_i)\|_{L_p}\mik \theta.
\end{equation}
We set $J=\{i+1,\dots,j\}$ and we claim that the interval $J$ is as desired. To this end notice, first, that $J^{\mathsf{c}}\neq\emptyset$.
Moreover, by \eqref{e3.3} and the choice of $c(\ee,p)$ and $\theta$ in \eqref{e1.3} and \eqref{e3.11} respectively, we have
\begin{equation} \label{e3.13}
|J|=j-i\meg \big( 4^{-1} \theta^2 (p-1) \big)\, n = c(\ee,p) \, n.
\end{equation}
Next, let $I$ be a nonempty subset of $J$ and set
\begin{equation} \label{e3.14}
g=\ave(f \, | \, \cals_j) \ \text{ and } \ h=\ave(f \, | \, \cals_i).
\end{equation}
We have the following claim.
\begin{claim} \label{c7}
For every $\bx\in\bbo_I$ we have $\ave(g_{\bx})=\ave(f_{\bx})$ and $\ave(h_{\bx})=\ave(f)$.
\end{claim}
\begin{proof}[Proof of Claim \emph{\ref{c7}}]
Fix $\bx\in \bbo_I$ and set $\mathcal{I}=\{1,\dots,i\}$ and $\mathcal{J}=\{1,\dots,j\}$.

First we argue to show that $\ave(g_{\bx})=\ave(f_{\bx})$. Indeed, observe that $I\subseteq J\subseteq \mathcal{J}$ and so,
by \eqref{e3.14} and Fubini's theorem, we see that for every $\by\in \bbo_{\mathcal{J}\setminus I}$ the function
$g_{(\bx,\by)}\colon \bbo_{\mathcal{J}^{\mathsf{c}}}\to \rr$ is constant and equal to $\ave(f_{(\bx,\by)})$. Therefore,
\begin{eqnarray} \label{e3.15}
\ave(g_{\bx})=\int g_{\bx} \, d\bbp_{\! I^{\mathsf{c}}} & = &
\int \Big(\int g_{(\bx,\by)} \, d\bbp_{\! \mathcal{J}^{\mathsf{c}}} \Big) \, d\bbp_{\! \mathcal{J}\setminus I} \\
& = & \int \ave(f_{(\bx,\by)}) \, d\bbp_{\! \mathcal{J}\setminus I} \nonumber \\
& = & \int\Big(\int f_{(\bx,\by)} \, d\bbp_{\! \mathcal{J}^{\mathsf{c}}}\Big) \, d\bbp_{\! \mathcal{J}\setminus I} \nonumber \\
& = & \int f_{\bx} \, d\bbp_{\! I^{\mathsf{c}}}= \ave(f_{\bx}). \nonumber
\end{eqnarray}

We proceed to show that $\ave(h_{\bx})=\ave(f)$. As above we notice that, by \eqref{e3.14} and Fubini's theorem, for every
$\bz\in \bbo_{\mathcal{I}}$ the function $h_{\bz}\colon\bbo_{\mathcal{I}^{\mathsf{c}}}\to \rr$ is constant and equal to $\ave(f_{\bz})$.
Since $\mathcal{I}\cap I=\emptyset$, the function $h_{(\bx,\bz)}\colon\bbo_{(\mathcal{I}\cup I)^{\mathsf{c}}}\to \rr$ is also constant
and equal to $\ave(f_\bz)$. Hence,
\begin{eqnarray} \label{e3.16}
\ave(h_{\bx}) = \int h_{\bx} \, d\bbp_{\! I^{\mathsf{c}}} & = &
\int\Big(\int h_{(\bx,\bz)} \,  d\bbp_{\! (\mathcal{I}\cup I)^{\mathsf{c}}}\Big) \, d\bbp_{\! \mathcal{I}} \\
& = & \int \ave(f_{\bz}) \, d\bbp_{\! \mathcal{I}} = \ave(f) \nonumber
\end{eqnarray}
and the proof of Claim \ref{c7} is completed.
\end{proof}
By Claim \ref{c7}, for every $\bx\in\bbo_I$ we have
\begin{equation} \label{e3.17}
|\ave(f_{\bx})-\ave(f)| = \big| \int (g_{\bx}-h_{\bx}) \, d\bbp_{\! I^{\mathsf{c}}}\big| \mik \|g_{\bx}-h_{\bx}\|_{L_1}
\end{equation}
and so
\begin{equation} \label{e3.18}
\int |\ave(f_{\bx})-\ave(f)|^p \, d\bbp_{\! I} \mik \int \|g_{\bx}-h_{\bx}\|^p_{L_1} \, d\bbp_{\! I}.
\end{equation}
It follows by Lemma \ref{l6}, \eqref{e3.12}, \eqref{e3.14} and the previous estimate that
\begin{equation} \label{e3.19}
\int |\ave(f_{\bx})-\ave(f)|^p \, d\bbp_{\! I} \mik \theta^p.
\end{equation}
Therefore, by Markov's inequality, we conclude that
\begin{equation} \label{e3.20}
\bbp_{\! I} \big( \{ \bx\in\bbo_I: |\ave(f_{\bx})-\ave(f)| \meg \theta^{\frac{p}{p+1}} \} \big)
\mik \frac{\theta^p}{\theta^{p^2/(p+1)}} = \theta^{\frac{p}{p+1}}
\end{equation}
which is equivalent to saying, by the choice of $\theta$ in \eqref{e3.11}, that
\begin{equation} \label{e3.21}
\bbp_{\! I} \big( \{ \bx\in\bbo_I: |\ave(f_{\bx})-\ave(f)| \mik \ee \} \big) \meg 1-\ee.
\end{equation}
The proof of Theorem \ref{t1} is completed.
\end{proof}


\section{Comments}

\subsection*{4.1}

For every positive integer $n$ and every finite set $A$ with $|A|\meg 2$ let
\begin{equation} \label{e4.1}
A^n=\{(a_1,\dots,a_n): a_1,\dots, a_n\in A\}
\end{equation}
and let $\mathbb{P}$ be the uniform probability measure on the hypercube $A^n$. Moreover, for every nonempty
subset $I$ of $\{1,\dots,n\}$ by $\mathbb{P}_{\! A^I}$ we denote the uniform probability measure on
$A^I\coloneqq \{(a_i)_{i\in I}: a_i\in A \text{ for every } i\in I\}$. We have the following lemma.
\begin{lem} \label{l8}
Let $k,m$ be positive integers with $k\meg 2$ and $0<\eta\mik 1$. Also let $A$ be a set with $|A|=k$ and let $n$ be a positive integer with
\begin{equation} \label{e4.2}
n\meg \frac{16m k^{3m}}{\eta^3}
\end{equation}
Then for every subset $D$ of $A^n$ there exists an interval $I\subseteq \{1,\dots,n\}$ with $|I|=m$ such that for every $t\in A^I$ we have
\begin{equation} \label{e4.3}
|\mathbb{P}_{\! A^{I^{\mathsf{c}}}}(D_t)-\mathbb{P}(D)|\mik \eta
\end{equation}
where  $D_t=\{s\in A^{I^{\mathsf{c}}}: (t,s)\in D\}$ is the section of $D$ at $t$.
\end{lem}
A simpler version of Lemma \ref{l8} was proved in \cite{DKT1} and was used as a tool in a proof of the density Hales--Jewett
theorem \cite{FK}; closely related applications were also obtained in \cite{DKT2} (see also \cite{DK}). Of course, the main point
in Lemma \ref{l8} is that by demanding a large---but not necessarily dense---set $I$ of coordinates, one can upgrade Theorem \ref{t1}
and guarantee that the probability of \textit{every} section of $D$ along elements of $A^I$ is essentially equal to the probability of $D$.
We proceed to the proof.
\begin{proof}[Proof of Lemma \emph{\ref{l8}}]
We view $A$ and $A^n$ as discrete probability spaces equipped with their uniform probability measures. Then notice that the
probability space $A^n$ is the product of $n$ many copies of $A$. Next we set $\ee = \eta k^{-m}2^{-1/3}$ and we observe that,
by \eqref{e1.3} and \eqref{e4.2}, we have $n\meg m\big(2/c(\ee,2)\big)$. Hence, by Corollary \ref{c2} applied to the set $D$,
the constant $\ee$ and $p=2$, there exists an interval $J\subseteq \{1,\dots,n\}$ with $J^{\mathsf{c}}\neq \emptyset$ and
$|J|\meg 2m$, and satisfying \eqref{e1.7} for every nonempty $I\subseteq J$. We select an interval $I\subseteq J$ with $|I|=m$
and we claim that $I$ is as desired. Indeed, by the choice of $\ee$, we have
\begin{equation} \label{e4.4}
\mathbb{P}_{\! A^I}\big( \{ t\in A^I: |\mathbb{P}_{\! A^{I^{\mathsf{c}}}}(D_t)-\mathbb{P}(D)|\mik \ee \} \big) \meg 1-\ee
\meg 1- k^{-m}2^{-1/3} > 1-\frac{1}{|A^I|}
\end{equation}
which implies that $|\mathbb{P}_{\! A^{I^{\mathsf{c}}}}(D_t)-\mathbb{P}(D)|\mik \ee$ for every $t\in A^I$. Since $\ee\mik\eta$
we conclude that the estimate in \eqref{e4.3} is satisfied and the proof is completed.
\end{proof}

\subsection*{4.2}

There is a natural extension of Theorem \ref{t1} which deals simultaneously with a family of random variables.
Although in applications one usually encounters only finite families of random variables (see, e.g., \cite{DKT2}),
the cleanest formulation of this extension is for stochastic processes indexed by the sample space of a probability
space $(T,\Sigma,\mu)$. Specifically, we have the following theorem.
\begin{thm} \label{t9}
Let $0<\ee\mik 1$ and $1<p\mik 2$, and set
\begin{equation} \label{e4.5}
c'(\ee,p) = \frac{1}{4}\, \ee^{\frac{2(2p+1)}{p}} (p-1).
\end{equation}
Also let $n$ be a positive integer with $n\meg 2/c'(\ee,p)$ and let $(\bbo,\bcalf,\bbp)$ be the product of a finite sequence
$(\Omega_1,\calf_1,\mathbb{P}_1),\dots,(\Omega_n,\calf_n,\mathbb{P}_n)$ of probability spaces. Finally, let $(T,\Sigma,\mu)$
be a probability space and $F\colon T\times \bbo\to \rr$ a random variable with $\|F\|_{L_p}\mik 1$. Then there exist
$G\in\Sigma$ with $\mu(G)\meg 1-\ee$ and an interval $J$ of $\{1,\dots,n\}$  with $J^{\mathsf{c}}\neq\emptyset$ and
\begin{equation} \label{e4.6}
|J|\meg c'(\ee,p)\, n
\end{equation}
such that for every $t\in G$ and every nonempty $I\subseteq J$ we have
\begin{equation} \label{e4.7}
\mathbf{P}_{\!I} \big( \{ \mathbf{x}\in\boldsymbol{\Omega}_I: |\ave(F_{(t,\bx)}) - \ave(F_t)|\mik \ee \}\big) \meg 1-\ee.
\end{equation}
\end{thm}
The proof of Theorem \ref{t9} is similar to the proof of Theorem \ref{t1} and so we will briefly sketch the argument.
First, for every $m\in\{1,\dots,n\}$ we define
\begin{equation} \label{e4.8}
\widetilde{\cals}_m= \Sigma\otimes \cals_m
\end{equation}
where $\cals_m$ is as in \eqref{e3.1}. Each $\widetilde{\cals}_m$ is a sub-$\sigma$-algebra of $\Sigma\otimes\bcalf$ and, moreover,
the finite sequence $(\widetilde{\cals}_m)_{m=1}^n$ is increasing. Hence, by Lemma \ref{l5} applied to $F$,
the filtration $(\widetilde{\cals}_m)_{m=1}^n$ and $\theta=\ee^{(2p+1)/p}$, there exist $i,j\in \{1,\dots,n-1\}$ with
\begin{equation} \label{e4.9}
j-i \meg \big( 4^{-1}\, \theta^2 (p-1)\big)\, n \stackrel{\eqref{e4.5}}{=} c'(\ee,p)\, n
\end{equation}
and such that
\begin{equation} \label{e4.10}
\| \ave(F \, | \, \widetilde{\cals}_j) - \ave(F \, | \, \widetilde{\cals}_i)\|_{L_p} \mik \ee^{\frac{2p+1}{p}}.
\end{equation}
Set $g=\ave(F \, | \, \widetilde{\cals}_j)$ and $h=\ave(F \, | \, \widetilde{\cals}_i)$, and notice that,
by Lemma \ref{l6} and \eqref{e4.10},
\begin{equation} \label{e4.11}
\int \|g_t-h_t\|_{L_1}^p \, d\mu \mik \ee^{2p+1}.
\end{equation}
Therefore, by Markov's inequality, there exists $G\in\Sigma$ with $\mu(G)\meg 1-\ee$ such that
\begin{equation} \label{e4.12}
\|g_t-h_t\|_{L_1}\mik \ee^2
\end{equation}
for every $t\in G$. The set $G$ and the interval $J\coloneqq \{i+1,\dots,j\}$ are as desired. Indeed, let $I$ be a nonempty subset of $J$.
Observe that $g_t=\ave(F_t \, | \, \cals_j)$ and $h_t=\ave(F_t \, | \, \cals_i)$ for every $t\in T$ which implies, by Claim \ref{c7}, that
$\ave(g_{t,\bx})=\ave(F_{t,\bx})$ and $\ave(h_{t,\bx})=\ave(F_t)$ for every $t\in T$ and every $\bx\in\bbo_I$. Taking into account these
observations, we conclude that the estimate in \eqref{e4.7} follows from \eqref{e4.12} and a second application of Markov's inequality.

\subsection*{4.3}

Recall that a Banach space $X$ is said to be \textit{uniformly convex} if for every $\ee>0$ there exists $\delta>0$ such that for
every $x,y\in X$ with $\|x\|_X=\|y\|_X=1$ and $\|x-y\|_X\meg \ee$ we have that $\|(x+y)/2\|_X\mik 1-\delta$. A classical result due to
James \cite{J} and, independently, V. Gurarii and N. Gurarii \cite{GG}, implies that for every uniformly convex Banach space $X$ and
every $p>1$ there exist $q\meg 2$ and a constant $C>0$ such that for every $X$-valued martingale difference sequence $(d_i)_{i=1}^n$ we have
\begin{equation} \label{e4.13}
\Big( \sum_{i=1}^n \|d_i\|_{L_p(X)}^q \Big)^{1/q} \mik C \big\| \sum_{i=1}^n d_i \big\|_{L_p(X)}.
\end{equation}
(See, also, \cite{Pi} for a proof and a detailed presentation of related material.) Using this estimate and arguing precisely as in Section 3,
we obtain the following vector-valued version of Theorem \ref{t1}.
\begin{thm} \label{t10}
For every uniformly convex Banach space $X$, every $0<\ee\mik 1$ and every $p>1$ there exists a constant $c(X,\ee,p)>0$ with the following
property. Let $n$ be a positive integer with $n\meg c(X,\ee,p)^{-1}$ and let $(\bbo,\bcalf,\bbp)$ be the product of a finite sequence
$(\Omega_1,\calf_1,\mathbb{P}_1),\dots,(\Omega_n,\calf_n,\mathbb{P}_n)$ of probability spaces. If $f\colon\bbo\to X$ is a random variable
with $\|f\|_{L_p(X)}\mik 1$, then there exists an interval $J$ of $\{1,\dots,n\}$ with $J^{\mathsf{c}}\neq\emptyset$ and
\begin{equation} \label{e4.14}
|J| \meg c(X,\ee,p)\, n
\end{equation}
such that for every nonempty $I\subseteq J$ we have
\begin{equation} \label{e4.15}
\mathbf{P}_{\!I} \big( \{ \mathbf{x}\in\boldsymbol{\Omega}_I: \|\ave(f_{\bx}) - \ave(f)\|_X\mik \ee \}\big) \meg 1-\ee.
\end{equation}
\end{thm}


\end{document}